\documentclass[reqno]{amsart}
\usepackage{amsmath,amssymb}
\usepackage{amsmath,amssymb,cite}
\usepackage{graphics,epsfig,color}
\usepackage[mathscr]{euscript}

\usepackage{xcolor}

\definecolor{bblue}{rgb}{.2,0.2,.8}

\theoremstyle{plain}
\newtheorem{theorem}{Theorem}[section]
\newtheorem{proposition}[theorem]{Proposition}
\newtheorem{lemma}[theorem]{Lemma}

\theoremstyle{definition}

\newtheorem{assumption}[theorem]{Assumption}

\theoremstyle{remark}

\numberwithin{equation}{section}
\numberwithin{theorem}{section}

\newcommand{\bb}[1]{{\mathbb #1}}
\newcommand{\ms}[1]{{\mathscr #1}}

\renewcommand{\epsilon}{\varepsilon}

\DeclareMathOperator{\tr}{Tr}

\title[Donsker-Varadhan meet Freidlin-Wentzell]{
  Large deviations for diffusions:\\ Donsker-Varadhan meet Freidlin-Wentzell 
  }

\author{L. Bertini}
\address{Lorenzo Bertini \hfill\break \indent
   Dipartimento di Matematica, Universit\`a di Roma `La Sapienza'
   \hfill\break \indent   P.le Aldo Moro 2, 00185 Roma, Italy}
 \email{bertini@mat.uniroma1.it}

\author{D. Gabrielli}
\address{\noindent Davide Gabrielli \hfill\break\indent 
 DISIM, Universit\`a dell'Aquila
\hfill\break\indent 
67100 Coppito, L'Aquila, Italy
}
\email{davide.gabrielli@univaq.it}

\author{C. Landim}
\address{Claudio Landim
  \hfill\break\indent IMPA \hfill\break\indent Estrada Dona Castorina
  110, \hfill\break\indent
J. Botanico, 22460 Rio de Janeiro, Brazil\hfill\break\indent
  {\normalfont and} \hfill\break\indent CNRS UMR 6085, Universit\'e de
  Rouen, \hfill\break\indent Avenue de l'Universit\'e, BP.12,
  Technop\^ole du Madril\-let, \hfill\break\indent
F76801 Saint-\'Etienne-du-Rouvray, France.} 
\email{landim@impa.br}

\begin{document}

\dedicatory{To Errico Presutti for his constant help and encouragement} 

\begin{abstract}
  We consider a diffusion process on $\bb R^n$ and prove a large
  deviation principle for the empirical process in the joint limit in
  which the time window diverges and the noise vanishes. The
  corresponding rate function is given by the expectation of the
  Freidlin-Wentzell functional per unit of time. As an application of
  this result, we obtain a variational representation of the rate
  function for the Gallavotti-Cohen observable in the small noise and
  large time limits. 
\end{abstract}

\noindent


\maketitle
\thispagestyle{empty}

\section{Introduction}
\label{s:0}

A diffusion processes on $\bb R^n$ can be realized as the solution to the
stochastic differential equation
\begin{equation}
  \label{2.2}
  \begin{cases}
    d\xi^\epsilon_t = b(\xi^\epsilon_t) dt + \sqrt{2\epsilon}\,
    \sigma(\xi^\epsilon_t)  dw_t\\
    \xi^\epsilon_0=x
  \end{cases}
\end{equation}
where $b$ is a smooth vector field, $w$ is a standard
$m$-dimensional Brownian, $\sigma$ is a $n\times m$ matrix
field, and the parameter $\epsilon>0$, that can be interpreted as the
temperature of the environment, will eventually vanish.
We shall impose conditions on $b$ and $\sigma$ which ensure the
ergodicity of the process $\xi^\epsilon$.

An \emph{additive functional} $\{A_T\}_{T\ge 0}$ of $\xi^\epsilon$ is
a real-valued, progressively measurable, functional of $\xi^\epsilon$
vanishing at $T=0$ and such that $A_{T+S} = A_T + A_S \circ \theta_T$,
where $\theta_T$ denotes the translation by $T$. Readily, functions of
the occupation measure, i.e.\ functional of the form
\begin{equation}
  \label{l2}
  A_T= \int_0^T \!dt\, f(\xi^\epsilon_t), \qquad f\colon \bb R^n\to \bb R,
\end{equation}
are examples of additive functionals. The basic question that we here
address is the behavior of additive functionals in the joint limit in
which the time window $[0,T]$ diverges and the noise $\epsilon$ vanishes. 
More precisely, we establish a large deviation principle in such
joint limit. According to the specific system modeled by \eqref{2.2}
and the details of the experimental setting, both the regimes
$\epsilon\ll T^{-1}$ and $\epsilon\gg T^{-1}$ are relevant.

According to the Donsker-Varadhan ideology \cite{dv1-4}, rather than
focusing on a single additive functional, the large deviation
principle is better formulated for a whole family of additive
functionals. This is formally realized by analyzing the
asymptotics of the \emph{empirical process} and the corresponding 
large deviations are usually called at \emph{level three}. Of course,
the rate function for a specific additive functional can then be
obtained by projecting the level three rate function.

To purse the joint limit $T\to \infty$ and $\epsilon\to 0$ there are
two simple alternatives.  (i) By taking first the limit
$\epsilon\to 0$ the large deviations of the empirical process can be
obtained by lifting the Freidlin-Wentzell asymptotic \cite{fw} to the
set of translation invariant probabilities on the path space. The
limit as $T\to \infty$ is then achieved by analyzing the variational
convergence of the corresponding, $T$-dependent, rate function.  (ii)
By taking first the limit $T \to \infty$ the large deviations of the
empirical process are directly given by the level three
Donsker-Varadhan asymptotic \cite{dv1-4}.  The limit as
$\epsilon\to 0$ is then achieved by analyzing the variational
convergence of the corresponding, $\epsilon$-dependent, rate function.
We here follow both these alternative and show they lead to the same
conclusion, the resulting rate function being particularly simple to
describe: it is the expectation of the Freidlin-Wentzell rate function
per unit of time.  When the deterministic dynamical system obtained by
setting $\epsilon=0$ in \eqref{2.2} has not a unique attractor, as it
is the case for metastable processes, this large deviation rate
function has not a unique zero. Therefore higher order large
deviations asymptotics can be investigated. For these asymptotics, the
order of the limit procedure $\epsilon\to 0$ and $T\to \infty$ becomes
relevant. We refer to \cite{dgm,bgl2,la,la2} for the corresponding analysis in the
context of reversible processes when the limit $\epsilon\to 0$ is
taken after $T\to\infty$.

In the context of non-equilibrium statistical mechanics, a relevant
additive functional not of the form \eqref{l2} is the Gallavotti-Cohen
observable \cite{gc,ku1,ls,ma1}. As we here discuss, its large
deviations in joint limit in which the time window diverges and the
noise vanishes can be obtained by projection.

The analysis here performed shares common features with the one
carried out in \cite{bgl} for the weakly asymmetric exclusion process
in the hydrodynamic scaling limit.
The present setting avoids the technicalities involved in hydrodynamic
limits and the core of the argument is more transparent.  On the other
hand, the non-compactness of the state space requires additional
estimates.

\section{Notation and main result}
\label{s:1}

We denote by $\cdot$ the canonical inner product in $\bb R^n$ and by
$|\cdot|$ the corresponding Euclidean norm. For $\epsilon>0$ we
consider the diffusion process on $\bb R^n$ with generator
$L_\epsilon$ defined on $C^2$ functions on $\bb R^n$ with compact
support  by
\begin{equation}
  \label{2.1}
  L_\epsilon f =\epsilon \tr ( a D^2 f) + b\cdot \nabla f
\end{equation}
where $D^2f$, respectively $\nabla f$, denotes the Hessian,
respectively the gradient, of $f$ and
$a= \{a_{i,j}(\cdot),\,i,j=1,\ldots,n\big\}$, respectively 
$b=\{b_i(\cdot),\,i=1,\ldots,n\big\}$, are the diffusion matrix
and the drift.
We suppose that the vector field $b$ admits the decomposition
\begin{equation}
  \label{2.0}
  b = - a \nabla V + c.
\end{equation}
Hereafter, we assume without further mention that $a,V,c$ meet the
following conditions in which we denote by $\bb M_n$ the set of
symmetric $n\times n$ matrices.

\begin{assumption}$~$
  \label{t:2.0}
  \begin{itemize}
  \item[(i)] $V$ belongs to $C^2(\bb R^n)$, $V\ge 0$, 
    $
    \displaystyle{
      \lim_{|x|\to \infty}  \nabla V (x) \cdot \frac x{|x|} =+\infty
    }$, and there exists $\epsilon_0>0$ such that 
    \begin{equation*}
      \displaystyle{
      \lim_{|x|\to \infty} \Big[ \nabla V(x)\cdot a(x) \nabla V(x)
      - \epsilon_0 \tr\big(a(x) D^2 V(x) \big)\Big]=+\infty;}
    \end{equation*}
  \item [(ii)] $c$ belongs to $C^1(\bb R^n;\bb R^n)$ and it is bounded
    with bounded derivatives;
  \item [(iii)] $a$ belongs to $C^2(\bb R^n;\bb M_n)$, it is bounded
    with bounded derivatives, and it is uniformly elliptic,
    i.e., there is constant $C>0$ such that
    $v\cdot a(x)v \ge  C^{-1} |v|^2$ for any $x,v\in \bb
    R^n$. 
\end{itemize}
\end{assumption}

The process generated by $L_\epsilon$ and initial condition
$x\in \bb R^n$ can be realized as the solution to the stochastic
differential equation \eqref {2.2} choosing $\sigma$ a globally
Lipschitz matrix field satisfying $a=\sigma\sigma^\dagger$.
In the present context, the vector field $b$ is not necessary globally
Lipschitz; however Assumption~\ref{t:2.0} implies there exists a
unique strong solution to \eqref{2.2}, see e.g.\ \cite[Thm.~3.5]{has}.
We shall denote the law of $\xi^\epsilon$ by $\bb P^\epsilon_x$ that,
given $T>0$, we regard as a probability on $C([0,T];\bb R^n)$.

We denote by $D(\bb R;\bb R^n)$ the space of c\`adl\`ag paths with
values on $\bb R^n$ that we consider endowed with the Skorokhod
topology on bounded intervals and the associated Borel
$\sigma$-algebra.  Given $T>0$ and a path $X\in C([0,T];\bb R^n)$ we
denote by $X^T\in D(\bb R;\bb R^n)$ its $T$-\emph{periodization},
i.e.,
\begin{equation*}
  (X^T)_t := X_{t-\lfloor t/T\rfloor T}, \qquad t\in \bb R. 
\end{equation*}
Observe that $X^T$ is $T$-periodic and continuous except at the times
$k T$, $k\in \bb Z$ where it has the jump of size $X_0-X_T$.  For
$t\in \bb R$ we denote by
$\theta_t \colon D(\bb R;\bb R^n)\to D(\bb R;\bb R^n)$ the translation
by $t$ namely, $(\theta_t X)_s:= X_{s-t}$, $s\in \bb R$. We finally
denote by $\ms P_\theta$ the set of translation invariant
probabilities on $D(\bb R;\bb R^n)$, i.e.\ the set of Borel
probabilities $P$ satisfying $P\circ \theta_t^{-1} = P$ for any
$t\in \bb R$. We consider $\ms P_\theta$ endowed with the topology
induced by weak convergence and the associated Borel $\sigma$-algebra.

Given $T>0$, the \emph{empirical process} is the map 
$R_T \colon C([0,T];\bb R^n) \to \ms P_\theta$ defined by
\begin{equation}
  \label{2.3}
  R_T (X) :=  \frac 1T \int_0^T\!dt\, \delta_{\theta_t X^T}.
\end{equation}
Note indeed that, by the $T$-periodicity of $X^T$, the right hand side
defines a translation invariant probability on $D(\bb R;\bb R^n)$.

Our main result establishes the large deviation principle for the
family of probabilities on $\ms P_\theta$ given by
$\big\{\bb P^\epsilon_x\circ R_T^{-1}\big\}$ in the joint limit
$\epsilon\to 0$ and $T\to \infty$.
Let us first recall the
Freidlin-Wentzell functional associated to \eqref{2.2}.
Given $T>0$, denote by $H_1=H_1([0,T])$ the set of absolutely
continuous paths $X\colon [0,T]\to \bb R^n$ such that $\int_0^T\!dt\,
|\dot X_t|^2 <+\infty$ and let $I_{[0,T]}\colon C([0,T];\bb R^n)\to
\bb [0,+\infty]$ be the functional defined by 
\begin{equation}
  \label{2.4}
  I_{[0,T]}(X) :=
  \begin{cases}
    \displaystyle{
      \! \frac 14 \!\int_{0}^T \!\!\! dt 
      \big[\dot X_t - b(X_t)\big]\cdot a^{-1}(X_t)
      \big[\dot X_t - b(X_t)\big]}
    &\!\!\!\textrm{if $X\in H_1$,}
    \\
    +\infty &\!\!\!
    \textrm{otherwise.}
  \end{cases}
\end{equation}
We regard $I_{[0,T]}$ as a functional on $D(\bb R;\bb R^n)$
understating that $I_{[0,T]}(X)$ is infinite if the restriction of $X$
to $[0,T]$ does not belong to $C([0,T];\bb R^n)$.  We then let
$\ms I\colon \ms P_\theta\to [0,+\infty]$ be the functional defined by
\begin{equation}
  \label{2.5}
  \ms I(P) := \int \! dP(X)\, I_{[0,1]}(X).
\end{equation}
Observe that $\ms I$ is affine and, by the translation invariance of
$P$, if  $\ms I(P) <+\infty$ then $P$-a.s. $t\mapsto X_t$ is
absolutely continuous.  
In the next statement we use the shorthand notation
$\varlimsup_{T,\epsilon}$ for either
$\varlimsup_{\epsilon\to 0} \varlimsup_{T\to \infty}$ or
$\varlimsup_{T\to\infty} \varlimsup_{\epsilon\to 0}$. Analogously,
$\varliminf_{T,\epsilon}$ stands for either
$\varliminf_{\epsilon\to 0} \varliminf_{T\to \infty}$
or $\varliminf_{T\to \infty} \varliminf_{\epsilon\to 0}$.

\begin{theorem}
\label{t:1}
As $\epsilon\to 0$ and $T\to \infty$,
the family $\big\{ \bb P^\epsilon_x \circ R_T^{-1},\,
T>0,\,\epsilon>0\big\}$ satisfies, uniformly for $x$ in compacts, a
large deviation principle with speed $\epsilon^{-1}T$ and rate
function $\ms I$. Namely, for each compact set $K\subset\subset \bb R^n$, 
each closed set $C \subset \ms P_\theta$,
and each open set $A\subset \ms P_\theta$ 
\begin{equation*}
  \begin{split}
    \varlimsup_{T,\epsilon}\;
    \sup_{x \in K} \frac \epsilon T
    \log \bb P^\epsilon_x \big( R_{T} \in C\big)
    \le -\inf_{P \in C} \ms I (P) \;,
    \\
    \varliminf_{T,\epsilon}\;
    \inf_{x \in K} \frac \epsilon T
    \log \bb P^\epsilon_x \big( R_{T} \in A\big)
    \ge -\inf_{P \in A} \ms I (P) \;.
\end{split}
\end{equation*}
Moreover, the functional $\ms I$ is good and affine.
\end{theorem}

Referring to Section~\ref{s:4} for an application of this result to
the asymptotics of the Gallavotti-Cohen observable, we next mention
some of its possible developments. While Theorem~\ref{t:1} suggests
that the large deviations hold whenever $(\epsilon,T)\to (0,+\infty)$,
the proof relies in computing first the limit as $T\to \infty$ and then
$\epsilon\to 0$ or the converse. It thus appears that a truly joint
limit requires new methods.  In the case in which the limiting
deterministic dynamical system obtained by setting $\epsilon=0$ in
\eqref{2.2} has more than a single stationary probability, as it is
the case for metastable processes, the zero level set of the
functional $\ms I$ is not a singleton.  In the spirit of the so-called
development by $\Gamma$-convergence, see e.g.\ \cite[\S~1.10]{Br}, it
is then possible to investigate higher order large deviations
asymptotics. In the case of reversible diffusions, this development
for the Fisher information, i.e.\ the Donsker-Varadhan level two rate
function for the occupation measure, has been achieved in
\cite{dgm}. The corresponding analysis for finite state Markov chains
has been carried out in \cite{bgl2,la,la2}. We emphasize that the limits
as $T\to \infty$ and $\epsilon\to 0$ do not commute for the higher
order large deviations. While the present analysis is carried out for
non-degenerate diffusion processes, the problem of computing the small
noise limit of the level three Donsker-Varadhan functional can be
formulated for general Markov processes.
According to \eqref{2.3}, the empirical process has been
defined in terms of the $T$-periodization of the path.  While this
choice is not relevant for the statements in Theorem~\ref{t:1}, it
will affect the higher order large deviations.

\section{Large time limit after small noise limit}
\label{s:2}

Recalling \eqref{2.4}, for $T>0$ and $x\in \bb R^n$ let
$I^x_{[0,T]} \colon C([0,T];\bb R^n)\to [0,+\infty]$ be the functional
defined by
\begin{equation}
  \label{3.1}
  I^x_{[0,T]}(X):=
  \begin{cases}
    I_{[0,T]} (X) & \textrm{if $X_0=x$,}\\
    +\infty &\textrm{otherwise.}
  \end{cases}
\end{equation}
Let also
$\ms I^x_{[0,T]} \colon \ms P_\theta \to [0,+\infty]$ be defined by
\begin{equation}
  \label{3.2}
  \ms I^x_{[0,T]} (P) := \inf\big\{  I^x_{[0,T]} (X),\; 
  R_T(X) = P \big\}\;,
\end{equation}
where we adopt the standard convention that the infimum over the empty
set is $+\infty$. Note that, for $X\in C([0,T];\bb R^n)$, if
$X(0) \not = X(T)$ or if $X(0) = X(T) = x$, then
$\ms I^x_{[0,T]} (P) = I^x_{[0,T]} (X)$. In contrast, if $X(0) = X(T)$
and $X(0) \not = x$, $I^x_{[0,T]} (X) = +\infty$ and
$\ms I^x_{[0,T]} (P)$ may be finite if $X(t) =x$ for some
$0\le t\le T$.  In view of the continuity of the map
$C([0,T];\bb R^n)\ni X \mapsto R_T(X) \in \ms P_\theta$, the following
statement follows directly, by the contraction principle, from the
Freidlin-Wentzell asymptotics \cite{fw}.  The present case of an
unbounded vector field $b$ is covered by \cite[Thm.~III.2.13]{az}.

\begin{lemma}
  \label{t:2.1}
  Fix $T>0$. As $\epsilon\to 0$ the family
  $\big\{ \bb P^\epsilon_x \circ R_T^{-1},\, \epsilon>0\big\}$
  satisfies, uniformly for $x$ in compacts, a large deviation
  principle with speed $\epsilon^{-1}$ and good rate function
  $\ms I^x_{[0,T]}$.
  Namely, for each $x\in \bb R^n$, each sequence
  $x_\epsilon\to x$, each closed set $C \subset \ms P_\theta$, and each
  open set $A\subset \ms P_\theta$
  \begin{equation*}
  \begin{split}
    \varlimsup_{\epsilon\to 0}\;
    \epsilon 
    \log \bb P^\epsilon_{x_\epsilon} \big( R_{T} \in C\big)
    \le -\inf_{P \in C} \ms I^x_{[0,T]} (P)
    \\
    \varliminf_{\epsilon\to 0}\;
    \epsilon 
    \log \bb P^\epsilon_{x_\epsilon} \big( R_{T} \in A\big)
    \ge -\inf_{P \in A} \ms I^x_{[0,T]} (P).
\end{split}
\end{equation*}
\end{lemma}

In order to achieve the proof of Theorem~\ref{t:1} we next analyze the
variational convergence of the family of functionals
$\big\{ T^{-1} \,\ms I^x_{[0,T]}\big\}$ as $T\to \infty$.  With
respect to the standard framework of $\Gamma$-convergence, see e.g.\
\cite{Br}, in the present setting there is the additional dependence on
the parameter $x$, for which we need uniformity on compacts.

\begin{theorem}
  \label{t:2.2}
  Fix a compact set $K\subset\subset \bb R^n$.
  \begin{itemize}
  \item [(i)] If a sequence $\{P_T\}\subset \ms P_\theta$ satisfies
    $\varliminf_T T^{-1} \,\ms I^{x_T}_{[0,T]} (P_T) < +\infty$ for
    some $\{x_T\}\subset K$ then $\{P_{T}\}$ has a pre-compact
    sub-sequence.
  \item[(ii)] For any $P\in \ms P_\theta$, any sequence
    $\{x_T\} \subset K$, and any sequence $P_T\to P$
    \begin{equation*}
      \varliminf_{T\to \infty} \frac 1T \ms I^{x_T}_{[0,T]}(P_T) \ge
      \ms I (P).
    \end{equation*}
  \item[(iii)] For any $P\in \ms P_\theta$ and any sequence
    $\{x_T\}\subset K$ there exists a sequence $P_T\to P$ such that
    \begin{equation*}
      \varlimsup_{T\to \infty} \frac 1T \ms I^{x_T}_{[0,T]}(P_T) \le
      \ms I (P).
    \end{equation*}
  \end{itemize}
\end{theorem}

Assuming the above result, we first show that it implies the large
deviations of the empirical process in the limit in which first the
noise vanishes and then the time interval diverges.

\begin{proof}[Proof of Theorem~\ref{t:1} ($T\to \infty$ after
  $\epsilon\to 0$).] 
  We start by showing the goodness of the rate function. Since
  $I_{[0,T]}$ is lower semi-continuous, by Portmanteau theorem, $\ms I$
  is also lower semi-continuous. It thus suffices to show that $\ms I$
  has pre-compact sublevel sets.
  In view of the conditions in Assumption~\ref{t:2.0}, by expanding
  the square in \eqref{2.4} we deduce there are constants $\gamma,C>0$
  depending only on $V,c,a$ such that for any $X\in C([0,T];\bb R^n)$
  \begin{equation}
    \label{3.3}
      I_{[0,T]}(X) \ge \frac 12 \big[ V(X_T) -V(X_0)\big] + \gamma
      \int_0^T\!dt \, \big[ |\dot X_t|^2 +|\nabla V(X_t)|^2 \big] -
      CT.
  \end{equation}
  Take expectation with respect to $P$. The translation
  invariance of $P$ and the bound $I_{[0,T+S]}(X) \le I_{[0,T]}(X) +
  I_{[0,S]}(\theta_{-T} X)$ yields that
  for each bounded interval $[T_1,T_2]$
  \begin{equation}
    \label{3.3bis}
    \int\! dP(X) \bigg[
    |\nabla V(X_0)|^2
    + \frac 1{T_2-T_1} \int_{T_1}^{T_2}\!dt\, |\dot X_t|^2 \bigg]
    \le C \big[ 1 + \ms I(P) \big]
    \end{equation}
    for a new constant $C$. By the assumptions on $V$ and standard
    criterion, see e.g.\ \cite[Thm.~8.2]{bi}, $\ms I$ has pre-compact
    sublevel sets, as claimed.

  To prove the upper bound, we first observe that the Feller property
  of the semigroup generated by $L_\epsilon$ and the continuity of
  $R_T$ imply that for each closed set $C\subset \ms P_\theta$ the map
  $x\mapsto \bb P^\epsilon_x (R_T\in C)$ is upper
  semi-continuous. Therefore, given a compact set
  $K\subset \subset \bb R^n$, there exists a sequence
  $\{x_{T,\epsilon}\}\subset K$ such that
  \begin{equation*}
    \sup_{x\in K} \bb P^\epsilon_x (R_T\in C) =
    \bb P^\epsilon_{x_{T,\epsilon}} (R_T\in C).
  \end{equation*}
  By passing to a not relabeled sub-sequence we may assume that the
  sequence $\{x_{T,\epsilon}\}_{\epsilon>0}$ converges to some
  $x_T\in K$. From Lemma~\ref{t:2.1} we then deduce
  \begin{equation*}
    \varlimsup_{\epsilon\to 0}  \sup_{x\in K}\epsilon\log \bb P^\epsilon_x
    \big(R_T\in C\big) \le -\inf_{P\in C} \ms I^{x_T}_{[0,T]}(P)
  \end{equation*}
  so that
  \begin{equation*}
    \varlimsup_{T\to \infty}
    \varlimsup_{\epsilon\to 0}  \sup_{x\in K}\frac \epsilon T
    \log \bb P^\epsilon_x
    \big(R_T\in C\big)
    \le -
    \varliminf_{T\to \infty} \inf_{P\in C} \frac 1T
    \ms I^{x_T}_{[0,T]}(P).
  \end{equation*}
  If
  $\varliminf_{T}\inf_{P\in C} T^{-1} \ms I^{x_T}_{[0,T]}(P)=+\infty$
  the right-hand side above is trivially bounded above by
  $-\inf_{P\in C} \ms I(P)$. If conversely
  $\varliminf_{T}\inf_{P\in C} T^{-1} \ms I^{x_T}_{[0,T]}(P)<
  +\infty$, there exist sequences $T_k\to\infty$ and
  $\{P_{k}\}\subset C$ such that
  \begin{equation*}
    \varliminf_{T\to \infty} \inf_{P\in C} \frac 1T
    \ms I^{x_T}_{[0,T]}(P)=
    \lim_{k\to \infty} \frac 1{T_k}
    \ms I^{x_{T_k}}_{[0,T_k]}(P_{k})\;.
  \end{equation*}
  By item (i) in Theorem~\ref{t:2.2}, there exists $P^*$ and a further
  sub-sequence, still denoted by $\{P_{k}\}\subset C$, converging to
  $P^*$. By the goodness of the rate function $\ms I^{x}_{[0,T]}$,
  $P^*\in C$, and, by item (ii) in Theorem~\ref{t:2.2},
  \begin{equation*}
    \lim_{k\to \infty} \frac 1{T_k}
    \ms I^{x_{T_k}}_{[0,T_k]}(P_{k}) 
    \ge \ms I(P^*)  \ge \inf_{P\in C} \ms I(P)
  \end{equation*}
  which concludes the proof of the upper bound.

  To prove the lower bound, observe that, again by the Feller property
  of the semigroup generated by $L_\epsilon$ and the continuity of
  $R_T$, for each open set $A\subset \ms P_\theta$ the map
  $x\mapsto \bb P^\epsilon_x (R_T\in A)$ is lower
  semi-continuous. Therefore, given a compact set
  $K\subset \subset \bb R^n$, there exists a sequence
  $\{x_{T,\epsilon}\}\subset K$ such that
  \begin{equation*}
    \inf_{x\in K} \bb P^\epsilon_x (R_T\in A) =
    \bb P^\epsilon_{x_{T,\epsilon}} (R_T\in A).
  \end{equation*}
  By passing to a not relabeled sub-sequence we may assume that the
  sequence $\{x_{T,\epsilon}\}_{\epsilon>0}$ converges to some
  $x_T\in K$. From Lemma~\ref{t:2.1} we then deduce
  \begin{equation*}
    \varliminf_{\epsilon\to 0}  \inf_{x\in K}\epsilon\log \bb P^\epsilon_x
    \big(R_T\in A\big) \ge -\inf_{P\in A} \ms I^{x_T}_{[0,T]}(P).
  \end{equation*}
  If $\inf_{P\in A} \ms I(P) = + \infty$, the right-hand side is
  bounded below by $- \inf_{P\in A} \ms I(P)$, and the lower
  bound of Theorem \ref{t:1} is proved. Conversely, assume that
  $\inf_{P\in A} \ms I(P) < + \infty$. In this case, given $\delta>0$,
  let $P^*\in A$ be such that
  $\inf_{P\in A} \ms I(P) \ge \ms I(P^*) -\delta$. By item (iii) in
  Theorem~\ref{t:2.2}, for $\{x_T\}\subset K$ as above there exists a
  sequence $\{P_T\}$ converging to $P^*$ and such that
  $\varlimsup_{T}T^{-1} \ms I^{x_T}_{[0,T]}(P_T)\le \ms I(P^*)$. Since
  $P^*\in A$, $P_T \to P^*$ and $A$ is an open set, $P_T \in A$ for
  $T$ large enough. Therefore,
  \begin{equation*}
    \begin{split}
    & \varliminf_{T\to\infty}
    \varliminf_{\epsilon\to 0}  \inf_{x\in K}\frac \epsilon T
    \log \bb P^\epsilon_x
    \big(R_T\in A\big) \ge
    - \varlimsup_{T\to\infty} \frac 1T \inf_{P\in A} \ms
    I^{x_T}_{[0,T]}(P)
    \\
    &\qquad\qquad
    \ge
    -\varlimsup_{T\to\infty} \frac 1T \ms I^{x_T}_{[0,T]}(P_T)
    \ge - \ms I(P^*) \ge - \inf_{P\in A} \ms I(P) -\delta,
  \end{split}
\end{equation*}
  which, by taking the limit $\delta\to 0$, concludes the proof.
\end{proof}

To prove Theorem~\ref{t:2.2}, we premise a density result on set of
translation invariant probability measures on $D(\bb R; \bb R^n)$.  An
element $P$ in $\ms P_\theta$ is said to be \emph{$S$-holonomic} if
there exists a $S$-periodic path $Y\in C(\bb R; \bb R^n)$ such that
\begin{equation}
  \label{16}
  P = \frac 1S\, \int_0^S \!dt\, \delta_{\theta_t Y},
\end{equation}
where we emphasize that we require $Y$ to satisfy $Y_S=Y_0$.
An element of $\ms P_\theta$ is \emph{holonomic} if it is
$S$-holonomic for some $S>0$; it is \emph{smooth holonomic} when the
path $Y$ in \eqref{16} belongs to $C^1(\bb R; \bb R^n)$.

\begin{lemma}
  \label{t:2.3}
  Fix $P\in \ms P_\theta$ satisfying $\ms I(P) <+\infty$.
  There exist a triangular array
  $\{\alpha^n_i,\, n\in \bb N,\, i=1,\ldots, n\}$ with
  $\alpha^n_{i}\ge 0$, $\sum_i \alpha^n_{i}=1$
  and a triangular array $\{P^n_{i}\,,\,n\in \bb N, i=1,\ldots, n\}$
  of smooth holonomic probability measures such that by setting
  $P^n := \sum_{i} \alpha^n_{i} P^n_{i}$ we have $P^n\to P$ and
  $\ms I (P^n) \to \ms I(P)$.
  \end{lemma}

\begin{proof}
  We follow the argument in \cite[Thm.~4.10]{bgl}, see also
  \cite{pb} for similar results.

  The proof is achieved, by a diagonal argument, from the following
  claims. Recall that $P\in \ms P_\theta$ is \emph{ergodic} when the
  tail $\sigma$-algebra is $P$-trivial.

 \smallskip
  \noindent\emph{Claim 1.}
  Let $P\in \ms P_\theta$ be such that $\ms I (P)<+\infty$.
  There exist a triangular array
  $\{\alpha^n_i,\, n\in \bb N,\, i=1,\ldots, n\}$ with
  $\alpha^n_{i}\ge 0$, $\sum_{i=1}^n \alpha^n_{i}=1$ and a triangular array
  $\{P^n_{i}\,,\,n\in \bb N, i=1,\ldots, n\}$ of ergodic
  probability measures such that
  $\sum_{i=1}^n \alpha^n_{i} P^n_{i}\to P$ 
  and $\sum_{i=1}^n \alpha^n_{i} \ms I(P^n_{i})\to \ms I(P)$.

  This follows directly from the fact that the ergodic probabilities
  are extremal in $\ms P_\theta$ and $\ms I$ is affine.

  \smallskip
  \noindent\emph{Claim 2.}
  Let $P\in \ms P_\theta$ be ergodic and such that $\ms I (P)<+\infty$.
  Then there exists a sequence $P^n\to P$ such that $\ms I(P^n)\to \ms
  I(P)$ and for each $n$ the probability $P^n$ is holonomic.

  Recalling \eqref{2.3}, to construct the required sequence set
  \begin{equation*}
    \begin{split}
      \ms A_P := \Big\{ X\in D(\bb R;\bb R^n) \colon
      \lim_{T\to \infty}\,
      R_T(X) = P \textrm{ and\,}
      \lim_{T\to \infty}
      \frac 1T I_{[0,T]} (X) =  \ms I(P)
      \Big\}.
    \end{split}
  \end{equation*}
  Since $\ms I(P)<+\infty$ then $I_{[0,1]} \in L_1(dP)$.
  The Birkhoff's ergodic theorem then implies $P(\ms A_P)=1$.
  Pick an element $Y\in \ms A_P$.  By definition, the $T$-holonomic
  probability associated to the $T$-periodization of $Y$ converges to
  $P$ but, in general, its rate function does not since when
  $T$-periodizing paths we may insert jumps.  This issue is easily
  solved by modifying the path $Y$ in the time interval $[T-1,T]$
  in such a way that $Y_T=Y_0$ and $T^{-1} I_{[T-1, T]}(Y)\to 0$.

 \smallskip
  \noindent\emph{Claim 3.}
  Let $P\in \ms P_\theta$ be holonomic and such that
  $\ms I (P)<+\infty$.  Then there exists a sequence of $C^1$
  holonomic probability measures $P_n\in \ms P_\theta$ such that
  $P^n\to P$ and $\ms I(P^n)\to \ms I(P)$.

  The required sequence is constructed by taking the convolution
  $\imath_n * X$ where $\imath_n$ is a smooth approximation of the
  identity and $X$ is the continuous periodic path associated to the
  measure $P$.
  \end{proof}

\begin{proof}[Proof of Theorem~\ref{t:2.2}.] \hfill\break
  \smallskip
  \noindent\emph{Item (i).} 
  By assumption, there exist a finite constant $C_0$ and sequences
  $T_j\to\infty$, $x_j\in K$, and $P_j \in \ms P_\theta$ such that
  $\ms I^{x_j}_{[0,T_j]}(P_j) \le C_0 T_j$. Fix $j$. By definition of
  $\ms I^{x_j}_{[0,T_j]}(P_j)$, there exists $Y\in C([0,T_j];\bb R^n)$
  satisfying $R_{T_j}(Y) = P_j$,
  $I^{x_j}_{[0,T_j]}(Y) \le \ms I^{x_j}_{[0,T_j]}(P_j) + 1$. As the
  rate function is finite, $Y(0)=x_j$.
  By
  \eqref{3.3} and since $V\ge 0$, 
  \begin{equation}
  \label{3.4}
  \begin{aligned}
   \int\! dP_j(X) \, |\nabla V (X_0)|^2  & = \frac{1}{T_j}
    \int_0^{T_j} \!dt |\nabla V (X_t)|^2 \\
    & \le \frac{1}{2T_j} V(x_j) +
    \frac 1{\gamma }
        \Big[ C + \frac 1{T_j} I^{x_j}_{[0,T_j]}(P_j)\Big] .
  \end{aligned}
  \end{equation}
  Since $I^{x_j}_{[0,T_j]}(Y) \le \ms I^{x_j}_{[0,T_j]}(P_j) + 1$ and
  $x_j$ belongs to a compact, the right-hand side is bounded by a finite
  constant, uniformly in $j$.

  The bound on the continuity modulus is somewhat more delicate as the
  $T$-periodization introduces, in general, jumps. On the other hand,
  given $T_1<T_2$ and $P=R_T(Y)$ for some
  $Y\in C([0,T];\bb R^n)$, the $P$ probability of observing a jump in
  the time window $[T_1,T_2]$ is at most $(T_2-T_1)/T$.  For
  $\delta>0$, $T_1<T_2$, and $X\in D(\bb R; \bb R^n)$, introduce the
  continuity modulus 
  \begin{equation*}
    \omega^{\delta}_{[T_1,T_2]} (X) := \sup_{\substack{t,s\in [T_1,T_2]\\
      |t-s|<\delta}} |X_t-X_s|.
  \end{equation*}
  By the Cauchy-Schwarz inequality, if the restriction of $X$ to $[T_1,T_2]$ belongs to
  $H^1([T_1,T_2])$ then 
  \begin{equation*}
    \omega^{\delta}_{[T_1,T_2]} (X)^2 \le
    \delta \int_{T_1}^{T_2}\!dt\, |\dot X|^2.
  \end{equation*}
  In view of the previous observations, if $P=R_T(Y)$ for some $Y$
  satisfying $I^x_{[0,T]}(Y) <+\infty$ for some $x\in K$, from
  Chebyshev inequality we deduce that for each $\zeta>0$
  \begin{equation*}
    \begin{split}
    P\big( \omega^{\delta}_{[T_1,T_2]} >\zeta \big)
    &\le \frac {T_2-T_1}{T} +\frac {(T_2-T_1)\delta}{\zeta^2}
    \frac 1T \int_0^T\!dt\, |\dot Y|^2
    \\
    &\le
    \frac {T_2-T_1}{T} +\frac {(T_2-T_1)\delta}{\gamma \zeta^2}
        \Big[ \frac 1{2T} \sup_{y\in K}V(y)+C
    + \frac 1T \ms I^x_{[0,T]}(P) \Big].
    \end{split}
  \end{equation*}
  where we used \eqref{3.3} in the second step.

  By standard criterion on tightness of probability measures on
  $D(\bb R;\bb R^n)$, see e.g.\ \cite[Thm.~15.5]{bi}, the previous displayed
  bound together with \eqref{3.4} yield the statement.

  \smallskip
  \noindent\emph{Item (ii).} 
  If $\ms I^{x}_{[0,T]}(P) <+\infty$ then there exists
  $Y\in C([0,T];\bb R^n)$ such that $P=R_T(Y)$ and for $T\ge 1$
  \begin{equation*}
    \begin{split}
      \ms I^{x}_{[0,T]} (P)
      & = I^{x}_{[0,T]} (Y) \ge I_{[0,T]} (Y)
      \ge \int_0^{T-1} \!\!dt \, I_{[0,1]}(\theta_{-t} Y)
      \\
      &= (T-1) \int\!d\tilde P(X)\, I_{[0,1]}(X)
    \end{split}
  \end{equation*}
  where we used \eqref{3.1} in the second step and we have set
  \begin{equation}
    \label{tildep}
    \tilde P := \frac{1}{T-1} \int_0^{T-1}\!dt \, \delta_{\theta_t Y^T}
    = \frac T{T-1} \, P - \frac 1{T-1} \,
    \int_{T-1}^T\!dt\, \delta_{\theta_t Y^T}.
  \end{equation}
  Consider now $P\in\ms P_\theta$ and sequences $\{x_T\}$, $P_T\to P$ as in
  the statement.
  By passing to a not relabeled sub-sequence we may
  assume that $P_T=R_T(Y)$ for some $Y=Y(T)\in C([0,T];\bb R^n)$.
  Letting $\tilde P_T$ be defined as in \eqref{tildep}
  we then have $\tilde P_T\to P$ and
  \begin{equation*}
    \begin{split}
          \varliminf_{T\to\infty}
     \frac 1T \ms I^{x_T}_{[0,T]} (P_T)
    &\ge \varliminf_{T\to\infty} \frac{T-1}{T}
    \, \int\!d\tilde P_T(X)\, I_{[0,1]}(X)
    \\
    &\ge \int\!dP(X)\, I_{[0,1]}(X) =\ms I(P)
    \end{split}  \end{equation*}
  where we have used the lower semi-continuity of $I_{[0,1]}$.

  \smallskip
  \noindent\emph{Item (iii).} 
  In view of Lemma~\ref{t:2.3}, it suffices to consider the case in
  which $P$ is smooth holonomic, i.e.
  $P = S^{-1} \int_0^S\!ds\, \delta_{\theta_s Y}$ for some
  $S>0$ and some $S$-periodic path $Y\in C^1(\bb R,\bb R^n)$.
  In particular, $\ms I(P) = S^{-1} I_{[0,S]}(Y)$.

  Given $x,y\in \bb R^n$ let ${\bar Y}^{x,y}\in C([0,1];\bb R^n)$ be
  the affine interpolation between $x$ and $y$, i.e.\
  ${\bar Y}^{x,y}_t = x(1-t)+yt$, $t\in [0,1]$. By a direct computation
  there exist $C(|x|,|y|)>0$ depending on $V,c,a$ such that
  \begin{equation*}
    I_{[0,1]}({\bar Y}^{x,y}) \le C(|x|,|y|).
  \end{equation*}
  For $T>0$ and a sequence $\{x_T\}\subset K$ as in the statement, let
  $\tilde Y\in C([0,+\infty);\bb R^n)$ be the path defined by
  \begin{equation*}
    \tilde Y_t :=
    \begin{cases}
      {\bar Y}^{x_T,Y_0}_t &\textrm{ if $t\in[0,1]$,}
      \\
      Y_{t-1} &\textrm{ if $t>1$,}
    \end{cases}
  \end{equation*}
  and set $P_T:= R_T(\tilde Y)$. Then $P_T\to P$ and for $T\ge 1$
  \begin{equation*}
    \ms I^{x_T}_{[0,T]} (P_T) = I^{x_T}_{[0,T]}(\tilde Y)
    = I_{[0,1]}\big({\bar Y}^{x_T,Y_0}\big)+ I_{[0,T-1]}(Y)
  \end{equation*}
  so that
  \begin{equation*}
    \varlimsup_{T\to \infty} \frac 1T \ms I^{x_T}_{[0,T]} (P_T)
    \le   \varlimsup_{T\to \infty}
    \Big[ \frac 1T \sup_{x\in K} C(|x|,|Y_0|) +  \frac 1T
    I_{[0,T-1]}(Y) \Big] = \ms I(P)
  \end{equation*}
  by the $S$-periodicity of $Y$.
\end{proof}

\section{Small noise limit after large time limit} 
\label{s:3}

By Assumption~\ref{2.1} and standard criteria, see
e.g.\ \cite[Thm.~3.7 and Cor.~4.4]{has}, for each $\epsilon>0$ the process
$\xi^\epsilon$ that solves \eqref{2.2} admits a unique invariant
probability $\pi^\epsilon$.
We denote by $\bb P^\epsilon_{\pi^\epsilon}$ the corresponding
stationary process, that we regard as a probability on
$D(\bb R;\bb R^n)$.  For fixed $\epsilon>0$, the Donsker-Varadhan
theorem \cite{deustr,dv1-4,Var84} states the large deviation principle as
$T\to\infty$ for the family
$\{ \bb P^\epsilon_x\circ R_T^{-1}\}_{T>0}$ with rate function given
by the relative entropy per unit of time with respect to
$\bb P^\epsilon_{\pi^\epsilon}$.

We first introduce such rate function by a variational
representation. For $T>0$, let
$\ms H^{\epsilon}(T, \cdot ) \colon \ms P_\theta \to [0,+\infty]$ be
the functional defined by
\begin{equation}
  \label{hev}
  \ms H^\epsilon (T, P) :=
  \sup_{\Phi} \int\! dP (X)
  \Big[  \Phi (X) - \log \bb E_{X_0}^\epsilon \big( e^{\Phi} \big) \Big],
\end{equation}
where $\bb E_{x}^\epsilon$ denotes the expectation with respect to 
$\bb P_{x}^\epsilon$, $x\in \bb R^n$ and 
the supremum is carried over the bounded and continuous
functions $\Phi$ on $ D(\bb R, \bb R^n)$ that are measurable with
respect to $\sigma\{ X_s, \: s\in[0,T]\big\}$.
Let then $\ms H^{\epsilon} \colon \ms P_\theta \to [0,+\infty]$ be
the functional defined by
\begin{equation} 
\label{20}
\ms H^{\epsilon}(P) := \sup_{T>0} \frac 1T \,
\ms H^\epsilon (T, P) = 
\lim_{T\to\infty} \frac 1T \, \ms H^\epsilon (T,P),
\end{equation}
where the second identity follows from the inequality before
\cite[Thm.~10.9]{Var84}.  By \cite[Thm.s~10.6 and 10.8]{Var84},
the functional $\ms H_{\epsilon}$ is good and affine.

We next characterize $\ms H^\epsilon$ as the relative entropy per unit
of time with respect $\bb P^\epsilon_{\pi^\epsilon}$.  Given
$T_1<T_2$, denote by
$\imath_{T_1,T_2}\colon D(\bb R, \bb R^n) \to D([T_1,T_2], \bb R^n)$
the canonical projection. Given two probability measures $P^1,P^2$,
let $\ms H_{[T_1,T_2]}(\cdot|\cdot)$ be the relative entropy of the
marginal of $P^2$ on the time interval $[T_1,T_2]$ with respect to the
marginal of $P^1$ on the same interval, i.e.,
\begin{equation}
  \label{rem}
  \ms H_{[T_1,T_2]} (P^2 | P^1 ) =
  \textrm{Ent} \big( P^2_{[T_1,T_2]} \big| P^1_{[T_1,T_2]} \big)
  := \int \!dP^2_{[T_1,T_2]} \,
  \log  \frac{dP^2_{[T_1,T_2]}} {dP^1_{[T_1,T_2]} }
\end{equation}
where $P^j_{[T_1,T_2]} = P^j \circ \imath_{T_1,T_2}^{-1}$, $j=1,2$.
By \cite[Thm.~5.4.27]{deustr}, for each $P\in \ms P_\theta$
\begin{equation}
\label{13}
\ms H^\epsilon (P) 
= \lim_{T\to \infty}  \frac 1T
\ms H_{[0,T]}\big( P \big| \bb P^\epsilon_{\pi^\epsilon} \big)
=  \sup_{T>0} \frac 1T\,
\ms H_{[0,T]}\big( P \big| \bb P^\epsilon_{\pi^\epsilon} \big),
\end{equation}
where the second identity follows by a super-additivity argument which
stems from \cite[Lemma 10.3]{Var84}.

Recalling that $\epsilon_0>0$ is the constant appearing in item (i) of
Assumption~\ref{t:2.0}, the large deviation principle in the limit
$T\to \infty$ is then stated as follows.

\begin{lemma}
  \label{t:3.1}
  Fix $\epsilon\in (0,\epsilon_0)$. As $T\to \infty$ the family
  $\big\{ \bb P^\epsilon_x \circ R_T^{-1},\, T>0\big\}$ satisfies,
  uniformly for $x$ in compacts, a large deviation principle with
  speed $T$ and good affine rate function $\ms H^\epsilon$.  Namely,
  for each compact set $K\subset\subset\bb R^n$, each closed set
  $C \subset \ms P_\theta$, and each open set $A\subset \ms P_\theta$
  \begin{equation*}
    \begin{split}
      \varlimsup_{T\to \infty}\; \frac 1T \sup_{x\in K} \log \bb
      P^\epsilon_{x} \big( R_{T} \in C\big) \le -\inf_{P \in C} \ms
      H^\epsilon(P)
      \\
      \varliminf_{T\to \infty}\; \frac 1T \inf_{x\in K} \log \bb
      P^\epsilon_{x} \big( R_{T} \in A\big) \ge -\inf_{P \in A}
      \ms H^\epsilon(P).
    \end{split}
  \end{equation*}
\end{lemma}

\begin{proof}
  The statement follows from \cite[Thm.s~11.6 and 12.5]{Var84}, we
  only need to check that the hypotheses of those theorems are met.

  Regarding the upper bound, given $\gamma\in(0,1)$ set
  \begin{equation}
    \label{ue}
    u_\epsilon(x):=\exp\Big\{ \frac \gamma\epsilon  V(x) \Big\},
    \qquad x\in \bb R^n.
  \end{equation}
  We claim that Assumption~\ref{t:2.0} implies that $u_\epsilon$ meets
  conditions (1)--(5) in \cite[Pag.~34]{Var84} for any
  $\epsilon\in (0,\epsilon_0)$ and a suitable $\gamma\in (0,1)$.
  Indeed, $u_\epsilon\ge 1$ and $u_\epsilon$ is bounded on
  compacts. Moreover, by a direct computation,
  \begin{equation}
    \label{we}
    W_\epsilon := - \frac {L_\epsilon u}{u} =
    \frac \gamma\epsilon\Big[
    (1-\gamma)\nabla V \cdot a\nabla V - c\cdot \nabla V
    - \epsilon \tr(a D^2 V) \Big]
  \end{equation}
  satisfies $\inf_{x} W_\epsilon (x) > - \infty$ and
  $\lim_{|x|\to \infty} W_\epsilon(x) = +\infty$ for $\gamma$ small
  enough.  Even if $u_\epsilon$ does not really belong to the domain
  of the generator $L_\epsilon$, it is straightforward to introduce a
  cutoff function $\phi_n\colon \bb R^n\to (0,+\infty)$ such that
  $u_{\epsilon,n}:= u_\epsilon \, \phi_n$ belongs to the domain of
  $L_\epsilon$ for each $n\in \bb N$ and the sequence
  $\{u_{\epsilon,n},\,n\in \bb N\}$ satisfies conditions (1)--(5) in
  \cite[Pag.~34]{Var84}.

  Regarding the lower bound, denote by $p^\epsilon(t,x,\cdot)$, $t\ge 0$,
  $x\in \bb R^n$ the transition probability of the Markov process
  $\xi^\epsilon$ and by $\alpha$ the Lebesgue measure on $\bb R^n$.
  By standard parabolic regularity, $p^\epsilon(1,x,\cdot)$ satisfies
  conditions I--II in \cite[Pag.~34]{Var84}.
\end{proof}

In view of the argument presented in the previous section, the proof
of Theorem~\ref{t:1} is completed by the variational convergence of
$\epsilon\,\ms H^\epsilon$ to $\ms I$.
As the $x$-dependence has disappeared in the limit $T\to \infty$, the
following statement amounts to the standard $\Gamma$-convergence of
the sequence $\{\epsilon\,\ms H^\epsilon\}$, see e.g.\ \cite{Br},
together with the pre-compactness of sequences $\{P_\epsilon\}$ with
equi-bounded rate function.

\begin{theorem}
  \label{t:3.2} $~$
  \begin{itemize}
  \item [(i)] If a sequence $\{P_\epsilon\}\subset \ms P_\theta$
    satisfies
    $\varliminf_\epsilon \epsilon \,\ms H^\epsilon (P_\epsilon) <
    +\infty$ then it has a pre-compact sub-sequence.
  \item[(ii)] For any $P\in \ms P_\theta$ and any sequence
    $P_\epsilon\to P$
    \begin{equation*}
      \varliminf_{\epsilon\to 0} \epsilon \,\ms H^\epsilon(P_\epsilon) \ge
      \ms I (P).
    \end{equation*}
  \item[(iii)] For any $P\in \ms P_\theta$ there exists a sequence
    $P_\epsilon\to P$ such that
    \begin{equation*}
      \varlimsup_{\epsilon\to 0} \epsilon \,\ms H^\epsilon(P_\epsilon) \le
      \ms I (P).
    \end{equation*}
  \end{itemize}
\end{theorem}

We next prove separately the three statements, each one having
a preliminary lemma.

\begin{lemma}
  \label{t:3.3}
  The sequence $\{\bb P^\epsilon_{\pi^\epsilon}\}\subset \ms P_\theta$
  is exponentially tight, i.e., there exists a sequence of compact
  sets 
  $\ms K_\ell\subset\subset D(\mathbb R, \mathbb R^n) $ such that
  \begin{equation*}
    \lim_{\ell\to\infty} \varlimsup_{\epsilon\to 0} \epsilon\,
    \log \bb P^\epsilon_{\pi^\epsilon} \big( \ms
    K_\ell^\mathrm{c}\big) =-\infty.
  \end{equation*}
\end{lemma}

\begin{proof}
  We first show that $\{\pi^\epsilon\}$ is an exponentially tight
  family of probabilities on $\bb R^n$. Observe that, by ergodicity,
  $\pi^\epsilon =\lim_{T\to \infty} T^{-1}\int_0^T\!dt\,
  \bb P_0^\epsilon(X_t\in \cdot)$. 
  Recalling \eqref{ue}, for $R>0$ let $u_\epsilon^R\colon \bb R^n \to
  [1,+\infty)$ be a smooth function such that 
  \begin{equation*}
    u_\epsilon^R(x) :=
    \begin{cases}
      u_\epsilon(x) &\textrm{if $|x|\ge R+1$,}
      \\
      1   &\textrm{if $|x|\le R$.}
    \end{cases}
  \end{equation*}
  In view of Assumption~\ref{t:2.0} and \eqref{we}, there are
  $R,\alpha>0$ such that for any $\epsilon$ small enough
  $L_\epsilon u_\epsilon^R \le -\alpha u_\epsilon^R$ so that
  \begin{equation*}
    \bb E_0^\epsilon \big(  u_\epsilon^R(X_t) \big)
    \le 1 -\alpha \int_0^t\!ds\, \bb E_0^\epsilon \big(
    u_\epsilon^R(X_s) \big).  
  \end{equation*}
  Whence, by Gronwall's lemma,
  $\sup_{t} \bb E^\epsilon_0\big(u_\epsilon^R(t) \big)\le 1$. 
  By changing the value of the parameter $\gamma\in (0,1)$ in
  \eqref{ue}, this bound provides the uniform integrability of
  $u_\epsilon^R$ with respect to 
  $\big\{T^{-1}\int_0^T\!dt\, \bb P_0^\epsilon(X_t\in \cdot)\big\}_{T>0}$.
  Therefore, by ergodicity,
  \begin{equation*}
    \int\! d\pi^\epsilon(x)\, u_\epsilon^R(x)  =
    \lim_{T\to\infty} \frac 1T \int_0^T\!dt \,
    \bb E_0^\epsilon\big(  u_\epsilon^R(X_t) \big) \le 1 
  \end{equation*}
  which, by Chebyshev inequality, yields the exponential tightness
  of $\{\pi^\epsilon\}$.

  We now observe that the Freidlin-Wentzell asymptotics implies that
  for each $T>0$ the family $\{\bb P^\epsilon_x\}_{\epsilon>0}$ is
  exponentially tight on $C([0,T];\bb R^n)$ uniformly for $x$ in
  compacts.  Since
  $\bb P^\epsilon_{\pi^\epsilon} =\int \! d\pi^\epsilon(x)\, \bb
  P^\epsilon_x$, the statement follows.
\end{proof}

\begin{proof}[Proof of Theorem~\ref{t:3.2}, item (i).]
  Fix $T_1<T_2$.  By the basic entropy inequality, see e.g.\
  \cite[Prop.~A1.8.2]{kl}, and \eqref{13}, for any
  $P\in\ms P_\theta$ and any event $B$ on $D([T_1,T_2];\bb R^n)$
    \begin{equation*}
    P(B) \le \frac {\log 2 + (T_2-T_1) \, \ms H^\epsilon(P) }
    {\log \Big(1 + \big[ \bb P^\epsilon_{\pi^\epsilon} ( B)\big]^{-1}
      \Big)}.
  \end{equation*}
  The statement now follows from Lemma~\ref{t:3.3}.
\end{proof}

As just proven, sequences $\{P_\epsilon\}$ with equi-bounded rate
function admit cluster points. We next show they enjoy some
regularity.

\begin{lemma}
  \label{t:3.4}
  There is a constant $C>0$ depending on $V,c,a$ such that the following
  holds. 
  If $\{P_\epsilon\}\subset \ms P_\theta$ is a sequence converging to
  $P$ then  for any $T_1<T_2$
  \begin{equation*}
    \int\!dP(X) \bigg[ \big| \nabla V (X_0) \big|^2 + \frac 1{T_2-T_1}
    \int_{T_1}^{T_2}\!dt\, |\dot X_t|^2 \bigg]
    \le C \Big[1 + \varliminf_{\epsilon\to 0} \epsilon \,
    \ms  H^\epsilon(P_\epsilon) \Big].
  \end{equation*}
\end{lemma}

\begin{proof}
  In order to obtain the estimate on $\int\!dP(X)\, |\nabla V(X_0)|^2$,
  we first prove the following bound. There are constants
  $\gamma, C>0$ such that for any $T>0$
  \begin{equation}
    \label{b1pi}
    \varlimsup_{\epsilon\to 0} \epsilon \log
    \bb E^\epsilon_{\pi^\epsilon}
    \Big(\exp\Big\{ \frac\gamma\epsilon \int_0^T\!\!dt\,
    \nabla V(X_t) \cdot a(X_t)\nabla V(X_t) \Big\} \Big) \le C (1+T).
  \end{equation}
  For $\lambda\in (0,1)$ to be chosen later, let $M^\lambda$ be the
  $\bb P^\epsilon_x$ martingale given by
  \begin{equation*}
    M^\lambda_t := \frac \lambda\epsilon \int_0^t\nabla
    V(X_s)\cdot \big(dX_s-b(X_s)ds\big)
  \end{equation*}
  where we understand the It\^o integral.
  Its quadratic variation is
  \begin{equation*}
    \langle M^\lambda\rangle_t := \frac {2\lambda^2}\epsilon
    \int_0^t\!ds\,
    \nabla V(X_s)\cdot a(X_s)\nabla V(X_s).
  \end{equation*}
  Setting $\Phi^\lambda_T:= M^\lambda_T -(1/2)  \langle
  M^\lambda\rangle_T$ and recalling that $b=-a\nabla V +c$,
  from It\^o's formula we get
  \begin{equation*}
    \begin{split}
      \Phi^\lambda_T &
      = \frac \lambda\epsilon\bigg\{ V(X_T)-V(X_0) +
      \int_0^T\!dt\, \Big[ (1-\lambda) \nabla V (X_t)\cdot
      a(X_t)\nabla V(X_t)
      \\
      &\phantom{\frac\lambda\epsilon\Big\{ V(X_T)-V(X_0) + \int_0^T}
      -\epsilon\tr\big(a(X_t) D^2 V(X_t)\big) -\nabla V(X_t) \cdot
      c(X_t) \Big]\bigg\}
    \end{split}
  \end{equation*}
  Assumption~\ref{t:2.0} implies that for each
  $\sigma\in (0,1-\lambda)$ there is a constant $C_\sigma$ such that
  for any $\epsilon$ small enough
  \begin{equation*}
    \Phi^\lambda_T \ge
    \frac \lambda\epsilon\Big\{
    -V(X_0) - C_\sigma T + (1-\lambda-\sigma) \int_0^T\!dt\,
    \nabla V (X_t)\cdot a(X_t)\nabla V(X_t) \Big\}.
  \end{equation*}
  Hence, setting $\gamma:= \lambda\, (1-\lambda-\sigma) / 2$,
  \begin{equation*}
    \frac \gamma\epsilon
    \int_0^T\!dt\, \nabla V (X_t)\cdot a(X_t)\nabla V(X_t)
    \le
    \frac 12 \Phi^\lambda_T
    + \frac \lambda{2\epsilon} \big[ C_\sigma T +V(X_0) \big]
  \end{equation*}
  so that, by Cauchy-Schwarz, 
  \begin{equation*}
    \Big[ \bb E^\epsilon_{\pi^\epsilon} \Big(
    e^{  \frac \gamma\epsilon
      \int_0^T\!dt\, \nabla V (X_t)\cdot a(X_t)\nabla V(X_t) }\Big)
    \Big]^{2}
    \le e^{\frac{\lambda C_\sigma T}\epsilon}
    \,
    \bb E^\epsilon_{\pi^\epsilon} \big(e^{\Phi^\lambda_T}
    \big)
    \int\!d\pi^\epsilon\, e^{\frac\lambda\epsilon V}.
  \end{equation*}
  We deduce the bound \eqref{b1pi} by observing that
  $\bb E^\epsilon_{\pi^\epsilon} \big(e^{\Phi^\lambda_T} \big)=1$ and,
  as follows from the proof of Lemma~\ref{t:3.3}, that there exists
  $\lambda\in (0,1)$ for which
  \begin{equation*}
    \varlimsup_{\epsilon\to 0} \epsilon
    \log \int\!d\pi^\epsilon\, e^{\frac \lambda\epsilon V } < +\infty.
  \end{equation*}

  By the variational characterization of the
  relative entropy, for any $P_\epsilon\in \ms P_\theta$
  \begin{equation*}
    \begin{split}
      &\int\!dP_\epsilon(X)\, \int_0^T\!dt\, \nabla V (X_t)\cdot
      a(X_t)\nabla V(X_t)
      \\
      &\qquad \le \frac \epsilon\gamma \log \bb
      E^\epsilon_{\pi^\epsilon}\Big( e^{ \frac \gamma\epsilon
        \int_0^T\!dt\, \nabla V (X_t)\cdot a(X_t)\nabla V(X_t)} \Big)
      + \frac \epsilon\gamma \ms H_{[0,T]}\big(P_\epsilon\big|\bb
      P^\epsilon_{\pi^\epsilon}\big).
    \end{split}
  \end{equation*}
  If $P_\epsilon\to P$, by the translation invariance of $P$, Fatou's
  lemma, the previous bound, \eqref{13} and \eqref{b1pi},
  \begin{equation*}
    \begin{split}
      & \int\! dP(X) \, \nabla V(X_0)\cdot a(X_0) \nabla V(X_0)
      \\
      &\qquad =
      \int\! dP(X) \,\frac 1T \int_0^T\!dt\, \nabla V(X_t)\cdot a(X_t)
      \nabla V(X_t)
      \\
      &\qquad
      \le \varliminf_{\epsilon\to 0} \int\! dP_\epsilon(X) \, \frac
      1T\int_0^T\!dt\, \nabla V(X_t)\cdot a(X_t) \nabla V(X_t)
      \\
      &\qquad
      \le \frac {C}{\gamma} \big( 1 +\frac 1T\big) +\frac 1\gamma
      \varliminf_{\epsilon\to 0} \epsilon\, \ms
      H^\epsilon(P_\epsilon) \;.
    \end{split}
    \end{equation*}
    As the left-hand side does not depend on $\epsilon$, we may choose
    at the beginning a sequence $\epsilon_k$ which achieves the
    $\liminf$ on the right-hand side.  Since $a$ is uniformly
    elliptic, the first assertion of the Lemma is proved.

  In order to obtain the estimate on the derivative, we next prove the
  following bound. There are constants $\gamma_1,\gamma_2, C>0$ such
  that for any $T>0$ and any $v\in C^1([0,T];\bb R^n)$ with support in
  $(0,T)$
  \begin{equation}
    \label{b2pi}
    \varlimsup_{\epsilon\to 0} \epsilon \log
    \bb E^\epsilon_{\pi^\epsilon}
    \bigg(\exp\Big\{ \frac{\gamma_1}\epsilon
    \int_0^T\!\!dt\big[ \dot v_t \cdot X_t - \gamma_2 |v_t|^2 \big]
    \Big\} \bigg) \le C (1+T).
  \end{equation}
  For $\lambda>0$ to be chosen later, let $M^\lambda$ be the
  $\bb P^\epsilon_x$ martingale given by
  \begin{equation*}
    M^\lambda_t := - \frac {2\lambda}\epsilon
    \int_0^tv_s \cdot \big(dX_s-b(X_s)ds\big)
  \end{equation*}
  whose quadratic variation is
  \begin{equation*}
    \langle M^\lambda\rangle_t := \frac {8\lambda^2}\epsilon
    \int_0^t\!ds\,
    v_s\cdot a(X_s) v_s.
  \end{equation*}
  Set $\Phi^\lambda_T:= M^\lambda_T -(1/2) \langle M^\lambda\rangle_T$
  and recall $v_0=v_T=0$. Integrating by parts and using
  Assumption~\ref{t:2.0} we deduce there are constants $\gamma_2, C>0$
  such that 
  \begin{equation*}
    \frac \lambda\epsilon 
    \int_0^T\!\!dt\big[ \dot v_t \cdot X_t - \gamma_2 |v_t|^2 \big]
    \le \frac 12 \Phi^\lambda_T +\frac {C\lambda}{2\epsilon}
    \Big\{ T +  \int_0^T\!dt\, \nabla V (X_t)\cdot a(X_t)\nabla V(X_t)
    \Big\}.
  \end{equation*}
  By choosing $\lambda$ small enough and using
  $\bb E^\epsilon_{\pi^\epsilon} \big(e^{\Phi^\lambda_T} \big)=1$
  together with \eqref{b1pi} we thus achieve the proof of
  \eqref{b2pi} by Cauchy-Schwarz.

  Pick a family $\{v^k\}$ of paths in $C^{1}((0,T); \bb R^n)$ with
  compact support and dense in $L^2((0,T); \bb R^n)$.  Assume that
  $v^1=0$.  In view of \eqref{b2pi}, the variational characterization
  of the relative entropy, and a classical argument which allows to
  bound a maximum over a finite set in exponential estimates, there
  exists a constant $C>0$ such that for any $N\in \bb N$,
  \begin{equation*}
    \begin{split}
    & \varliminf_{\epsilon\to 0}
    \int\!d P_\epsilon(X)
    \max_{k\in\{1,\ldots,N\}}
    \int_0^T\!\!dt
    \big[ \dot v_t^k \cdot X_t - \gamma_2 \big|v_t^k\big|^2 \big]
    \\
    &\qquad \le C( 1 +T)
    \Big[1 +   \varliminf_{\epsilon\to 0} \epsilon \ms
    H^\epsilon(P_\epsilon) \Big].
  \end{split}
\end{equation*}
  Since $P_\epsilon\to P$ and $v^1=0$, from Fatou's lemma we deduce
  \begin{equation*}
    \int\!d P(X)\max_{k\in\{1,\ldots,N\}}
    \int_0^T\!\!dt
    \big[ \dot v_t^k \cdot X_t - \gamma_2 \big|v_t^k\big|^2 \big]
    \le C( 1 +T)
    \Big[1 +   \varliminf_{\epsilon\to 0} \epsilon \ms
    H^\epsilon(P_\epsilon) \Big]
  \end{equation*}
  whence, by monotone convergence,
  \begin{equation*}
    \int\!d P(X)
    \sup_{k\in\bb N} \int_0^T\!\!dt
    \big[ \dot v_t^k \cdot X_t - \gamma_2 \big|v_t^k\big|^2 \big]
    \le C_0( 1 +T)
    \Big[1 +   \varliminf_{\epsilon\to 0} \epsilon \ms
    H^\epsilon(P_\epsilon) \Big].
  \end{equation*}
  Since the family $\{v^k\}$ is dense in $L^2((0,T);\bb R^n)$ this
  estimate implies that $P$-a.s. $X$ belongs to $H_1([0,T])$ and, by the
  translation invariance of $P$, the second part of the bound in the
  statement. 
\end{proof}

\begin{proof}[Proof of Theorem~\ref{t:3.2}, item (ii).]
  For $\delta>0$ let $\imath_\delta$ be a smooth probability
  density on $\bb R$ with support contained in $(0,\delta)$. For
  $X\in D(\bb R; \bb R^n)$ let $\imath_\delta*X \in C^\infty(\bb R;\bb
  R^n)$ be defined by
  \begin{equation*}
    (\imath_\delta*X)_t := \int\!ds\,
    \imath_\delta(t-s) X_s,
  \end{equation*}
  where, by the support property of $\imath_\delta$, we can restrict
  the integral to $(t-\delta,t)$. In particular,
  $\frac d{dt} {\imath_\delta * X} =\imath_\delta'* X$.
  Given $w\in C\big(\bb R^n\times \bb R^n;\bb R^n)$ bounded, let
  $W_\delta$ be the $\bb R^n$-valued function on
  $\bb R\times D(\bb R;\bb R^n)$ defined by
  \begin{equation*}
    W_\delta(t,X) := \chi_\delta (t) \,
    w \big((\imath_\delta* X)_t, (\imath'_\delta* X)_t\big),
  \end{equation*}
  where $\chi_\delta\colon \bb R \to [0,1]$ is a smooth function
  satisfying $\chi_\delta(t)=0$ for $t\le \delta$
  and $\chi_\delta(t)=1$ for $t\ge 2\delta$.
  Note that, by construction, $W_\delta(t,\cdot)$ is continuous,
  bounded, and measurable with respect to the $\sigma$-algebra
  generated by $\{X_s,\, s\in[0,t]\}$.

  Consider now the $\bb P^\epsilon_x$--martingale
  $M^{\delta,\epsilon}$ defined by 
  \begin{equation*}
    M^{\delta,\epsilon}_t := \frac 1\epsilon
    \int_0^t W_\delta(s,X) \cdot \big( d X_s - b(X_s) ds\big)
  \end{equation*}
  whose  quadratic variation is
  \begin{equation*}
    \langle M^{\delta,\epsilon} \rangle_t =
    \frac 2\epsilon \int_0^t\!ds\, W_\delta(s,X)\cdot a(X_s) W_\delta(s,X).
  \end{equation*}
  Let finally $\Phi_{\delta,\epsilon}\colon D(\bb R;\bb R^n)\to \bb R$ be
  the $\sigma\{X_s,\, s\in [0,1]\}$ measurable function defined by 
  \begin{equation*}
    \Phi_{\delta,\epsilon} := M^{\delta,\epsilon}_1 -\frac 12 \langle
    M^{\delta,\epsilon} \rangle_1
  \end{equation*}
  and observe that
  $\bb E^\epsilon_x\big(e^{\Phi_{\delta,\epsilon}} \big)=1$, $x\in \bb
  R^n$.

  Even if $\Phi_{\delta,\epsilon}$ is neither continuous nor bounded,
  by a truncation procedure whose details are omitted, see e.g.\
  \cite[Lemma~6.2]{Var84} for a similar argument, we can take
  $\Phi=\Phi_{\delta,\epsilon}$ in the variational representation
  \eqref{hev}. If $\{P_\epsilon\}\subset \ms P_\theta$ is a sequence
  converging to $P$, by
  \eqref{20} and the regularity of $P$ in Lemma~\ref{t:3.4}
  we deduce that
  \begin{align*}
    &\varliminf_{\epsilon\to 0} \epsilon\, \ms H^\epsilon(P_\epsilon)
    \\
    &\;\ge \int\!\!dP(X)
    \int_0^1\!\!dt\Big[
    W_\delta(t,X) \cdot \big(\dot{X}_t - b(X_t)\big)
    - W_\delta(t,X) \cdot a(X_t) W_\delta(t,X) \Big].
  \end{align*}
  In view of Lemma~\ref{t:3.4} and dominated convergence, we can take
  the limit as $\delta\to 0$ inside the integrals on the right hand side
  above. We thus infer that for any bounded
  $w\in C(\bb R^n\times \bb R^n;\bb R^n)$
  \begin{align*}
    &\varliminf_{\epsilon\to 0} \epsilon \, \ms H^\epsilon(P_\epsilon)
    \\ &
    \ge \int\!\!dP(X)
    \!\int_0^1\!\!\!dt\Big[
    w(X_t,\dot X_t) \!\cdot\! \big(\dot{X}_t - b(X_t)\big)
    - w(X_t,\dot X_t) \!\cdot\! a(X_t) w(X_t,\dot X_t) \Big].
  \end{align*}
  Recalling \eqref{2.4} and \eqref{2.5} we conclude, using again
  Lemma~\ref{t:3.4} and dominated convergence, by considering a
  suitable sequence $\{w_n\}$ with $w_n$ bounded for each $n$ and
  converging pointwise to $w^*$ with
  $w^*(x,y)=(1/2)\, a(x)^{-1}[y-b(x)]$.
\end{proof}

In view of density result proven in Lemma~\ref{t:2.3}, in order to
construct the recovery sequence in item (iii) of Theorem~\ref{3.2} it
suffices to consider the case in which $P$ is smooth holonomic, i.e.\
$P = S^{-1} \int_0^S\!ds\, \delta_{\theta_s Y}$ for some $S>0$ and
some $S$-periodic path $Y\in C^1(\bb R,\bb R^n)$.  To construct the
sequence $\{P_\epsilon\}$ for such $P$, pick first
$U\colon \bb R^n\to \bb R$ such that: $U\in C^2(\bb R^n)$, the minimum
of $U$ is uniquely attained at $x=0$, the Hessian $D^2 U(0)$ is
strictly positive definite, and $U=V$ outside some compact set
$K\subset\subset\bb R^n$.  Consider now the non-autonomous stochastic
differential equation
\begin{equation}
  \label{dhe}
  \begin{cases}
    d\eta^\epsilon_t = \tilde b_\epsilon (t,\eta^\epsilon_t) dt
    + \sqrt{2\epsilon}\, \sigma(\eta^\epsilon_t-Y_t)  dw_t
    \\
    \eta^\epsilon_0=x
  \end{cases}
\end{equation}
where
\begin{equation}
  \label{tbe}
  \tilde b_\epsilon (t,x):=-a(x-Y_t) \nabla U\big(x-Y_t\big) +
  \epsilon\,\nabla\cdot a (x-Y_t) + \dot{Y}_t,
\end{equation}
in which $\nabla\cdot a$ is the vector field given by the divergence of
$a$, i.e.\ $(\nabla\cdot a)_i =\sum_j \partial_j a_{j,i}$.
Note that $\tilde b_\epsilon$ is $S$-periodic in the first variable.
Denote the law of $\eta^\epsilon$ by $\bb Q^\epsilon_x$ and let
$\mu^\epsilon$ be the probability on $\bb R^n$ whose density is
proportional to $\exp\{-U/\epsilon\}$.  Set finally
$\nu^\epsilon:= \mu^\epsilon(Y_0+\,\cdot)$ and
$\bb Q^\epsilon_{\nu^\epsilon}:= \int\!d\nu^\epsilon(x)\,\bb
Q^\epsilon_x$.

\begin{lemma}
  \label{t:3.5}
  The probability $\bb Q^\epsilon_{\nu^\epsilon}$ is invariant with
  respect to $\theta_{S}$. Furthermore $\bb Q^\epsilon_{\nu^\epsilon}\to
  \delta_Y$ as $\epsilon\to 0$ and for each $\epsilon\in (0,\epsilon_0)$
  there exist a constant $C_\epsilon$ such
  that for any $n\in \bb N$ and $s\in[0,S]$ 
  \begin{equation*}
    \begin{split}
        &\frac{\epsilon}{n}
        \, \ms H^\epsilon_{[0, n S]}
        \big(\bb Q^\epsilon_{\nu^\epsilon}\circ\theta_s^{-1} \big)
        \le \frac {C_\epsilon}n
        \\
        &\quad +
        \frac 14 \int\!\! d \bb Q^\epsilon_{\nu^\epsilon}(X)
        \int_0^S\!\!dt\,
        \big[ \tilde b_\epsilon(t,X_t) -b(X_t) \big] \cdot
    a^{-1}(X_t) \big[ \tilde b_\epsilon(t,X_t) -b(X_t) \big]. 
  \end{split}
\end{equation*}
\end{lemma}

\begin{proof}
  By direct computation $\bb Q^\epsilon_{\nu^\epsilon}$ is the law of
  $Y+\zeta^\epsilon$ where $\zeta^\epsilon$ is the stationary
  process associated to the autonomous stochastic differential
  equation
  \begin{equation*}
    d\zeta^\epsilon_t = \big[
    -a(\zeta^\epsilon_t) \nabla U\big(\zeta^\epsilon_t\big)
    + \epsilon\,\nabla\cdot a(\zeta^\epsilon_t)\big]dt
    + \sqrt{2\epsilon}\, \sigma(\zeta^\epsilon_t)  dw_t.
  \end{equation*}
  Observe indeed that $\zeta^\epsilon$ is reversible with respect to
  $\mu^\epsilon$. Since $Y$ is $S$-periodic and the law of
  $\zeta^\epsilon$ is translation invariant we deduce that
  $\bb Q^\epsilon_{\nu^\epsilon}$ is invariant with respect to
  $\theta_S$.  By the properties of $U$, we readily conclude that
  $\zeta^\epsilon$ converges to $0$ in probability and therefore that
  $\bb Q^\epsilon_{\nu^\epsilon}\to \delta_Y$.

  For notation simplicity, we prove the entropy bound only when $s=0$.
  Let $M^\epsilon$ be the $\bb P^\epsilon_x$ martingale given by
  \begin{equation*}
    M^\epsilon_t := \frac 1{2\epsilon} \int_0^t
    a^{-1}(X_s) \big[ \tilde b_\epsilon(s,X_s) -b(X_s) \big]
    \cdot \big( dX_s-b(X_s)ds \big)
  \end{equation*}
  whose quadratic variation is 
    \begin{equation*}
    \langle M^\epsilon \rangle_t := \frac 1{2\epsilon} \int_0^t\!ds\, 
    \big[ \tilde b_\epsilon(s,X_s) -b(X_s) \big] \cdot
    a^{-1}(X_s) \big[ \tilde b_\epsilon(s,X_s) -b(X_s) \big]. 
  \end{equation*}
  By Girsanov formula, for each $T>0$
  \begin{equation*}
    \frac {d \big(\bb Q^\epsilon_x\big)_{[0,T]}}
    {d \big(\bb P^\epsilon_x\big)_{[0,T]}}
    =\exp\Big\{  M^\epsilon_T -
    \frac 12\langle M^\epsilon \rangle_T \Big\}.
  \end{equation*}
  Using \cite[Thm.~VIII.1.7]{RY} we deduce
  \begin{equation*}
    \textrm{Ent}\Big( \big(\bb Q^\epsilon_x\big)_{[0,T]} \big|
    \big(\bb P^\epsilon_x\big)_{[0,T]} \Big)
    = \int\! d\bb Q^\epsilon_x\,
    \Big[ M^\epsilon_T - \frac 12\langle M^\epsilon \rangle_T \Big]
    = \frac 12 \int\! d \bb Q^\epsilon_x \, \langle M^\epsilon \rangle_T
  \end{equation*}
  which yields
  \begin{equation*}
    \ms H^\epsilon_{[0,T]} \big( \bb Q^\epsilon_{\nu^\epsilon} \big)
    = \textrm{Ent}(\nu^\epsilon|\pi^\epsilon)
    +\frac 12 \int\! d \bb Q^\epsilon_{\nu^\epsilon} \,
    \langle M^\epsilon \rangle_T. 
  \end{equation*}
  
  In view of the $\theta_S$ invariance of
  $\bb Q^\epsilon_{\nu^\epsilon}$, setting
  $C_\epsilon:=\textrm{Ent}(\nu^\epsilon|\pi^\epsilon)$, the stated
  bound follows once we show that $C_\epsilon$ is finite.  To this
  end, we first obtain a lower bound on the tail of $\pi^\epsilon$.
  Denote by $\rho^\epsilon$ the density of $\pi^\epsilon$ with respect
  to the Lebesgue measure, $d\pi^\epsilon = \rho^\epsilon \, dx$.  By
  Assumption~\ref{t:2.0} and standard results, $\rho^\epsilon$ is
  smooth, strictly positive, and solves the stationary Fokker-Planck
  equation
  \begin{equation*}
  \epsilon \sum_{i,j=1}^n \partial_i\partial_j
  \big( a_{i,j} \rho^\epsilon\big)
    - \sum_{i=1}^n \partial_i \big(b_i \rho^\epsilon\big) =0\;.
  \end{equation*}
  Set $v^\epsilon := \rho^\epsilon \exp\{\gamma V/\epsilon\}$ for
  some $\gamma>0$ to be chosen later; by direct computation it solves
  \begin{equation*}
    A_\epsilon v^\epsilon + h v^\epsilon =0 
  \end{equation*}
  where $A_\epsilon$ is the elliptic operator defined by
  \begin{equation*}
    A_\epsilon v :=\epsilon \tr( a D^2 v)
    -\big( b + 2\gamma a\nabla V -2\epsilon \nabla\cdot a\big) \cdot \nabla v
  \end{equation*}
  and
  \begin{equation*}
    h := \frac \gamma\epsilon
    \big[ b \cdot \nabla V  + \gamma \nabla V\cdot a \nabla V\big]
    - \gamma \tr(aD^2V)
    -2\gamma (\nabla\cdot a)\cdot  \nabla V 
    - \nabla\cdot b 
    +\epsilon \partial_i\partial_j a_{i,j}.
  \end{equation*}
  As follows from Assumption~\ref{t:2.0}, for each
  $\epsilon\in (0,\epsilon_0)$ there exist $\gamma, R>0$ such that
  $h(x) \ge 0$ for all $x\in \bb R^n$ such that $|x|\ge R$.  Let now
  $m_\epsilon := \inf\{v^\epsilon(x), \, |x|=R\}>0$ and set
  $u^\epsilon= m_\epsilon - v^\epsilon$. Then $u^\epsilon (x) \le 0$
  for $|x|=R$ and, by the positivity of $v^\epsilon$, we have
  $u^\epsilon(x) \le m_\epsilon$ for any $x\in \bb R$.
  Finally, by the choices of $\gamma$ and $R$, for $|x|>R$ the
  function $u^\epsilon$ solves
  \begin{equation*}
    A_\epsilon u^\epsilon = A_\epsilon (m_\epsilon -v^\epsilon) =
    - A_\epsilon v^\epsilon = h v^\epsilon \ge 0.
  \end{equation*}
  From the Phragm\`en-Lindelh\"of maximum principle, see
  \cite[Thm.~2.19]{prwei}, we then deduce $u^\epsilon(x)\le 0$,
  for all $x\in \bb R^n$ such that $|x|>R$. Hence
  $\rho^\epsilon(x) \ge m_\epsilon \exp\{-\gamma V(x)/\epsilon\}$ for
  $|x|\ge R$. As $\nu^\epsilon(dx)
  = Z_\epsilon^{-1} \exp\{- U(x-Y_0)/\epsilon\} dx$ with $Z_\epsilon$
  the appropriate normalization, we get
  \begin{equation*}
    \begin{split}
      \textrm{Ent}(\nu^\epsilon|\pi^\epsilon) &= \int\!
      d\nu^\epsilon(x) \, \log \frac {e^{-
          U(x-Y_0)/\epsilon}}{Z_\epsilon \rho^\epsilon(x)}
      \\
      &\le \int_{|x|\le R} \! d\nu^\epsilon(x) \, \log \frac
      {e^{-U(x-Y_0)/\epsilon}}{Z_\epsilon\rho^\epsilon(x)}
      \\
      &\; +\int_{|x|> R} \! d\nu^\epsilon(x)\, \Big[ \log \frac
      1{Z_\epsilon m_\epsilon} -\frac 1\epsilon U(x-Y_0) +
      \frac{\gamma}{\epsilon} V(x) \Big]
    \end{split}
  \end{equation*}
  which is bounded as $V$ has super-linear growth as $|x|\to \infty$
  and $U=V$ outside a compact. 
\end{proof}

\begin{proof}[Proof of Theorem~\ref{t:3.2}, item (iii).]
  By Lemma~\ref{t:2.3} it suffices to consider the case in which $P$ is
  smooth holonomic.
  For $P$ and $\bb Q^\epsilon_{\nu^\epsilon}$ as introduced before
  Lemma~\ref{t:3.5}, set
  \begin{equation*}
    P_\epsilon := \frac 1S \int_0^S \!ds\,\bb
    Q^\epsilon_{\nu^\epsilon}\circ \theta_{s}^{-1} 
  \end{equation*}
  that is translation invariant by the $\theta_S$ invariance of
  $\bb Q^\epsilon_{\nu^\epsilon}$.
  By Lemma~\ref{t:3.5}, the sequence $\{P_\epsilon\}$ converges to
  $P$.  Moreover, using also \eqref{13} and the convexity of the
  relative entropy,
  \begin{equation*}
    \epsilon \,\ms H^\epsilon ( P_\epsilon) 
    \le 
    \frac 1{4S}\int_0^S\!\!dt
    \int\!\! d \bb Q^\epsilon_{\nu^\epsilon}(X)\,
    \big[ \tilde b_\epsilon(t,X_t) -b(X_t) \big] \cdot
    a^{-1}(X_t) \big[ \tilde b_\epsilon(t,X_t) -b(X_t) \big]. 
  \end{equation*}

  Recalling \eqref{tbe}, since $\bb Q^\epsilon_{\nu^\epsilon}\to
  \delta_Y$ then $\tilde b_\epsilon(t,\cdot)$ converges in
  $\bb Q^\epsilon_{\nu^\epsilon}$-probability to $\dot Y_t$.
  As the marginal at time $t$ of  $\bb Q^\epsilon_{\nu^\epsilon}$ is
  equal to $\nu^\epsilon_t := \mu^\epsilon(Y_t+ \,\cdot)$ and $U=V$
  outside some compact, we obtain the needed uniform integrability 
  to infer
  \begin{equation*}
    \varlimsup_{\epsilon\to 0}
    \epsilon \,\ms H^\epsilon ( P_\epsilon) 
    \le  \frac 1{4S}\int_0^S\!\!dt\, \big[ \dot Y_t -b(Y_t) \big]
    \cdot a^{-1}(Y_t) \big[ \dot Y_t - b(Y_t) \big]
    = \ms I (P)    \,, 
  \end{equation*}
  which concludes the proof.
\end{proof}

\section{Large deviations of the Gallavotti-Cohen observable}
\label{s:4}

The Gallavotti-Cohen functional has been originally introduced in the
context of chaotic deterministic dynamical systems as the expansion
rate of the phase-space volume and it has been shown to satisfy the
so-called fluctuation theorem \cite{gc}.  The definition of this
functional for stochastic dynamics has been originally discussed in
\cite{ku1} and in more generality in \cite{ls,ma1};
we refer to \cite{jpr} for a review and to \cite{cl} for an
experimental check of the fluctuation theorem.

In the present context of non-degenerate
diffusion processes, introduce the time inversion as the involution
$\Theta \colon C(\bb R;\bb R^n) \to C(\bb R;\bb R^n)$ given by
$(\Theta X)_t:= X_{-t}$. Recalling that
$\bb P^\epsilon_{\pi^\epsilon}$ denotes the stationary process
associated to \eqref{2.2}, the Gallavotti-Cohen functional
is defined by
\begin{equation*}
  \widehat{W}^\epsilon_{[0,T]}  :=
  \frac {\epsilon}T \log \frac
    {d \big( \bb P^\epsilon_{\pi^\epsilon}\big)_{[0,T]}}
  {d \big( \bb P^\epsilon_{\pi^\epsilon}\circ \Theta^{-1}\big)_{[0,T]}}
\end{equation*}
where the subscript $[0,T]$ denotes the restriction of the probability
to that time interval. The factor $\epsilon$ has been inserted for
notation convenience when discussing the small noise limit
$\epsilon\to 0$.  Note that
$\bb E^\epsilon_{\pi^\epsilon} \big(
\widehat{W}^\epsilon_{[0,T]}\big)\ge 0$ and
this expectation equals, apart a factor $\epsilon$, 
the relative entropy per unit of time of
$\bb P^\epsilon_{\pi^\epsilon}$ with respect to
$\bb P^\epsilon_{\pi^\epsilon}\circ \Theta^{-1}$.

The content of the fluctuation theorem is the following. Assume that
the family of real random variables
$\{ \widehat{W}^\epsilon_{[0,T]} \}_{T>0}$ satisfies a large deviation
principle as $T\to \infty$ and denote by
$s_\epsilon\colon \bb R\to [0,+\infty]$ the rate function. Then the
odd part of $s_\epsilon$ is linear,
$s_\epsilon(q)-s_\epsilon(-q) = -\epsilon q$, where the factor
$\epsilon$ is due to the choice of the normalization. The physical
interpretation of the fluctuation theorem is that the ratio between
the probability of the events
$\{ \widehat{W}^\epsilon_{[0,T]} \approx q\}$ and
$\{ \widehat{W}^\epsilon_{[0,T]} \approx -q\}$ becomes fixed,
independently of the model, in the large time limit.

An informal computation based on the Girsanov formula shows that
\begin{equation}
  \label{wteg}
  \widehat{W}^\epsilon_{[0,T]} (X) = \frac 1T \int_0^T
  a(X_t)^{-1} b(X_t)\circ dX_t
  - \frac {\epsilon}T
  \log \frac {\rho^\epsilon(X_T)}{\rho^\epsilon(X_0)}
\end{equation}
where $\circ$ denotes the Stratonovich integral and
$\rho^\epsilon$ is the density of the invariant measure
$\pi^\epsilon$.
In the case of a compact state space, the standard route to obtain the
large deviation principle for the family
$\{\widehat{W}^\epsilon_{[0,T]}\}_{T>0}$ is the following \cite{ls}.
Neglect the second term on the right hand side of \eqref{wteg}, which
becomes irrelevant in the limit $T\to\infty$, and prove, by using
Girsanov and Feynman-Kac formulae together with the Perron-Frobenious
theorem, that the limit
\begin{equation}
    \label{lwteg}
  \Lambda_\epsilon(\lambda) := \lim_{T\to \infty}
  \frac 1T \log \bb E^\epsilon_{\pi^\epsilon}
  \Big( \exp\Big\{ \lambda \int_0^T a(X_t)^{-1}b(X_t)\circ
  dX_t\Big\}\Big), 
\end{equation}
exists for each $\lambda\in \bb R$
and it can be expressed as the maximal eigenvalue of a perturbed
generator. An application of the Gartner-Ellis theorem then yields the
large deviation principle while the fluctuation theorem follows from
the symmetry $\Lambda_\epsilon(\lambda)
=\Lambda_\epsilon(-\epsilon -\lambda)$. We refer to \cite[\S~5]{ls} for the
informal derivation of this symmetry in the context of diffusions
processes. 

As detailed in \cite{jps}, the route sketched above in general fails
in the present case of non-compact space state: it is neither possible
to neglect the second term on the right hand side of \eqref{wteg} nor
to prove the existence of the limit in \eqref{lwteg} for any
$\lambda\in \bb R$.  Following \cite{bdg,Ra,wxx} and recalling the
decomposition \eqref{2.0}, we here \emph{define} the Gallavotti-Cohen
observable by
\begin{equation}
  \label{wen}
  {W}_{[0,T]} (X) := \frac 1T \int_0^T a(X_t)^{-1} c(X_t)\circ dX_t,
\end{equation}
namely as the work done, in the metric defined by the diffusion
matrix, by the non-conservative part of the drift.  In contrast to
\eqref{wteg}, ${W}_{[0,T]}$ is an empirical observable namely, an
explicit functional of the sample path.  As shown in
\cite{bdg,Ra,wxx}, for each $\epsilon>0$ the family of probabilities
on $\bb R$ given by
$\{\bb P^\epsilon_{\pi^\epsilon} \circ (W_{[0,T]})^{-1}\}_{T>0}$
satisfies a large deviation principle and the corresponding rate
function $s_\epsilon$ satisfies the fluctuation theorem.  The present
purpose is to obtain a variational representation of this rate
function in the small noise limit $\epsilon\to 0$. This problem has
been originally addressed heuristically in \cite{ku2}. A mathematical
analysis has been carried out in \cite{bdg} when the limit
$\epsilon\to 0$ is taken before the limit $T\to \infty$ and the
limiting rate function is then expressed in terms of the
Freidlin-Wentzell rate functional. In the same scaling as in
\cite{bdg}, we here show that the limiting rate function is actually
independent of the limiting procedure.  This analysis complements the
one in \cite{Ra}, where the small noise limit of the rate function for
the Gallavotti-Cohen observable is carried out with a different
scaling, that can be seen as a next order asymptotic with respect to
the one performed.

Before discussing the Gallavotti-Cohen observable, we note that the
odd part, with respect to the involution $\Theta$, of the rate
function $\ms I$ in \eqref{2.5} is in fact expressed in terms of the
functional introduced \eqref{wen}. In this respect, the next statement
can be seen as a fluctuation theorem at the level of the empirical process.

\begin{proposition}
  \label{t:5.1}
  For any $P\in \ms P_\theta$ such that $\ms I(P) < +\infty$
  \begin{equation*}
    \ms I\big(P\circ \Theta^{-1}\big) -\ms I\big(P\big)
    = \int\! dP(X)\, W_{[0,1]}(X)
    = \int\! d P (X)\int_0^1\!\! dt \, a(X_t)^{-1} c(X_t) \cdot \dot{X}_t.
  \end{equation*}
\end{proposition}

\begin{proof}
  Recalling \eqref{3.3bis}, that provides the needed integrability
  conditions, the proof is simply achieved by using the decomposition
  \eqref{2.0} and expanding the square in \eqref{2.4}. Note indeed
  that the boundary term vanishes by translation invariance.
\end{proof}

In the next statement we employ the same convention on
$\varlimsup_{\epsilon,T}$ and $\varliminf_{\epsilon,T}$ as the one
used in Theorem~\ref{t:1}.

\begin{theorem}
  \label{t:5.2}
  Assume that 
  $|x|\le C \big(1+ \big|\nabla V(x)\big|^2 \big)$, $x\in \bb R^n$, 
  for some constant $C>0$.
  Then, as $\epsilon\to 0$ and $T\to \infty$, the family of probabilities on
  $\bb R$ given by 
  $\big\{ \bb P^\epsilon_x \circ (W_{[0,T]})^{-1},\, T>0,\,\epsilon>0\big\}$
  satisfies, uniformly for $x$ in compact sets, a large deviation
  principle with speed $\epsilon^{-1}T$ and rate function
  $s\colon \bb R\to [0,+\infty]$ given by
  \begin{equation*}
    s(q) = \inf \Big\{ \ms I(P),\,
    \int\! d P (X) \int_0^1\!dt\, a(X_t)^{-1}c(X_t) \cdot \dot X_t =q
    \Big\}.
  \end{equation*}
  Namely, for each compact set $K\subset\subset \bb R^n$, each closed
  set $C \subset \bb R$, and each open set $A\subset \bb R$
  \begin{equation*}
    \begin{split}
      \varlimsup_{T,\epsilon}\; \sup_{x \in K} \frac \epsilon T \log
      \bb P^\epsilon_x \big( W_{[0,T]} \in C\big) \le -\inf_{q \in C}
      s(q)
      \\
      \varliminf_{T,\epsilon}\; \inf_{x \in K} \frac \epsilon T \log
      \bb P^\epsilon_x \big( W_{[0,T]} \in A\big) \ge -\inf_{q \in A}
      s(q).
    \end{split}
  \end{equation*}
  Moreover, the function $s$ is good, convex, and satisfies the fluctuation
  theorem $s(-q) - s(q) = q$.
\end{theorem}

Since, as proven in Lemma~\ref{3.3}, the family of probabilities
$\{\pi^\epsilon\}_{\epsilon>0}$ is exponentially tight, the previous
statement also holds when $\bb P^\epsilon_{x}$ is replaced by the
stationary process $\bb P^\epsilon_{\pi^\epsilon}$.

\begin{proof}
  It is convenient to rewrite $W_{[0,T]}$ in \eqref{wen} in terms of
  the It\^o integral,
  \begin{equation*}
    {W}_{[0,T]} (X) =  \widetilde{W}_{[0,T]} (X)
    + \epsilon \, Z^1_T(X)
  \end{equation*}
  where
  \begin{equation*}
    \widetilde{W}_{[0,T]} (X):= \frac 1T \int_0^T a(X_t)^{-1} c(X_t)\cdot dX_t
  \end{equation*}
  and, by Assumption~\ref{t:2.0}, $Z^1_T(X)$ is bounded uniformly in
  $T$ and $X$ and therefore irrelevant for the large
  deviations.
  Recalling the definition of the empirical process in \eqref{2.3} we
  next observe that
  \begin{equation}
    \label{twfr}
    \int\! d R_T(X) \, \widetilde{W}_{[0,1]} (X)
    = \widetilde{W}_{[0,T]} (X) + \frac 1T \, Z^2_T(X)
  \end{equation}
  where $Z^2_T$ takes into account the jump inserted
  by the $T$-periodization, 
  \begin{equation*}
    Z^2_T(X) =  a^{-1}(X_T) \, c(X_T) \cdot [X_0-X_T] .
  \end{equation*}
  As we assumed $|x|\le C (1+|\nabla V(x)|^2\big)$, the bounds
  provided by \eqref{3.3bis} and Lemma~\ref{t:3.4} imply that also
  $T^{-1} Z^2_T(X)$ is irrelevant for the large deviations. Therefore
  \eqref{twfr} expresses the Gallavotti-Cohen observable as a function
  of the empirical process. However, as $\widetilde{W}_{[0,1]}$
  involves the It\^o integral, this function is not continuous.  By a
  truncation procedure that it is not detailed, see
  \cite[Lemma~6.2]{Var84} for a similar argument, we can however
  construct a continuous, exponentially good approximation of
  $\widetilde{W}_{[0,1]}$ and deduce the large deviation principle
  for $\tilde W_{[0,T]}$ by contraction principle from
  Theorem~\ref{t:1}.

  The convexity of the rate function $s$ readily follows from its
  definition while the fluctuation theorem is a corollary of
  Proposition~\ref{t:5.1}. 
\end{proof}

\bigskip
\subsection*{Acknowledgments}
We thank G.\ Di Ges\`u and  M.\ Mariani for useful discussions.
We are also grateful to R. Raqu\'epas for a discussion on the
relationship between our and his work.

\end{document}